\documentclass[10pt,a4paper]{article}
\usepackage[T1]{fontenc}
\usepackage[utf8]{inputenc}
\usepackage{authblk}
\usepackage{amsmath,amssymb,amsthm,esint,bm}
\usepackage{mathrsfs}
\usepackage{bookmark}
\usepackage{amsmath}
\allowdisplaybreaks[3]

\newtheorem{definition}{Definition}[section]
\newtheorem{theorem}[definition]{Theorem}

\theoremstyle{remark}
\newtheorem{remark}[definition]{Remark}
\numberwithin{equation}{section}

\title{Symmetry of Positive Solutions for the Fractional Schr$ \ddot{\textrm{o}}$dinger Equations with Choquard-type Nonlinearities}

\title{Symmetry of Positive Solutions for the Fractional Schr$ \ddot{\textrm{o}}$dinger Equations with Choquard-type Nonlinearities}

\author[a]{Xiaoya Huang}
\author[b]{Zhenqiu Zhang\thanks{Corresponding author.}}

\affil[a]{School of Mathematical Sciences, Nankai University, Tianjin, 300071, P.R. China}
\affil[b]{School of Mathematical Sciences and LPMC, Nankai University, Tianjin, 300071, P.R. China}

\date{\today}

\usepackage{hyperref}
\begin{document}
\maketitle
\footnotetext[1]{E-mail: 2120170065@mail.nankai.edu.cn (X. Huang), zqzhang@nankai.edu.cn (Z. Zhang).}
\begin{abstract}
This paper deals with the following fractional Schr$ \ddot{\textrm{o}}$dinger equations with Choquard-type nonlinearities
\begin{equation*}
\left\{\begin{array}{r@{\ \ }c@{\ \ }ll}
 (-\Delta)^{\frac{\alpha}{2}}u + u - C_{n,-\beta} \,(|x|^{\beta-n}\ast u^{p})\, u^{p-1}& = &0      & \mbox{in}\ \ \mathbb{R}^{n}\,, \\[0.05cm]

 u & > & 0 & \mbox{on}\ \  \mathbb{R}^{n},
\end{array}\right.
\end{equation*}
where $ 0< \alpha,\beta < 2, 1\leq p <\infty  \,\,and\,\, n\geq 2.   $
First we construct a decay result at infinity and a narrow
region principle for related equations. Then we establish the radial symmetry of positive
solutions for the above equation with the generalized direct method of moving planes.



\end{abstract}


\section{Introduction.}\label{section1}

In this paper, we study the fractional Schr$ \ddot{\textrm{o}}$dinger equations with Choquard-type nonlinearities as follows
\begin{equation}\label{Big1}
\left\{\begin{array}{r@{\ \ }c@{\ \ }ll}
 (-\Delta)^{\frac{\alpha}{2}}u + u - C_{n,-\beta} \,(|x|^{\beta-n}\ast u^{p})\, u^{p-1}& = &0      & \mbox{in}\ \ \mathbb{R}^{n}\,, \\[0.05cm]

 u & > & 0 & \mbox{on}\ \  \mathbb{R}^{n},
\end{array}\right.
\end{equation}
where $ 0< \alpha,\beta < 2, 1\leq p <\infty  \,\,and\,\, n\geq 2. $ The fractional Laplacian is a nonlinear nonlocal pseudo differential operators defined as
\begin{equation}\label{Frac-Laplacian}
\begin{array}{r@{\ \ }c@{\ \ }ll}
(-\Delta)^\frac{\alpha}{2} u(x) &:=& C_{n,\alpha} PV  \displaystyle{\int_{\mathbb{R}^n} \frac{u(x)-u(y)}{|x-y|^{n+\alpha}} dy}, \\
&=&C_{n,\alpha}\underset{\varepsilon \rightarrow \infty} {lim} \displaystyle{\int_{\mathbb{R}^n \backslash B_{\varepsilon}(x)} \frac{u(x)-u(y)}{|x-y|^{n+\alpha}} dy},
\end{array}
\end{equation}
here $C_{n,\alpha}$ is a normalization constant.
To make sense for the integrals, we require
$u \in C^{1,1}_{loc} \cap L_{\alpha} $,\,
and set
$$ L_{\alpha} = \{ u \in L^{1}_{loc}(R^{n}) \mid \int_{\mathbb{R}^n} \frac{|u(x)|}{1+|x|^{n+\alpha}} d x < \infty \}.$$

The Choquard equation arises in a variety of applications such as quantum mechanics, physics of laser beams, physics of multiple-particle systems, which we refer to\cite{Gro}. We remark that the nonlinear Choquard equation is a certain approximation to Hartree-Fock theory for a component plasma\cite{LiBa}. It is worth mentioning that \cite{MZ} considered the fractional Choquard equations with the nonhomogeneous term $(|x|^{\beta-n}\ast u^{p})\,u^{p-1}$. Afterwards, Ma and Zhang\cite{mz} extended their results to Choquard equations with fractional p-Laplacian.

The fractional Laplacian has attracted much attention in recent years. It has been widely applied in mathematical physics\cite{SX}, image processing\cite{GO}, finance\cite{CT}, and so on. For more backgrounds on the fractional Laplacian operator $(-\Delta)^{\frac{\alpha}{2}} $, we refer to\cite{CSS,BN,S}. Nevertheless, since the fractional Laplacian is nonlocal, many traditional methods that deal with the local operator cannot be applied directly. To overcome this difficulty, the pioneering work can be traced back to Caffarelli and Silvestre\cite{CaSi}. They reduced the nonlocal problem into a local one in higher dimensions by introducing the extension method. This method has been applied successfully in investigating the equations with $(-\Delta)^{s}$, a series of related problems have been studied from then on(cf. \cite{BCPS,CZ} and the references therein). In addition, another effective tool to handle the higher order fractional Laplacian is the method of moving planes in integral forms, which turns a given pseudo differential equations into their equivalent integral equations, we refer to \cite{CFY,CLO1,CLO2} for details. However, even though using the extension method or the integral equations method, sometimes we still need to impose conditions on the range of $s$ or some integrability conditions on the solutions additionally. Furthermore, the methods we mentioned above are not applicable to other kinds of nonlinear nonlocal operators, such as the fully nonlinear nonlocal operator and fractional p-Laplacian(p$\neq 2$).

Lately, Chen et al.\cite{CLL} developed a direct method of moving planes which can overcome the difficulties above. They obtained the nonexistence, symmetry and monoticity of positive solutions for existence semi-linear equations involving the fractional Laplacian. We are particularly interested in the method using in handling the nonlinear Schr$ \ddot{\textrm{o}}$dinger equation in the article. We also refer to \cite{MaZhao,QY} for related interesting results of Schr$ \ddot{\textrm{o}}$dinger equations.

Inspired by the literatures above, we naturally consider the properties of positive solutions for the fractional Schr$ \ddot{\textrm{o}}$dinger equations with Choquard-type nonlinearities based on the extended direct method of moving planes. In this paper, we construct a decay result at infinity and a narrow
region principle for the related equations first in order to obtain the symmetry of solutions to the problem\eqref{Big1}.

Throughout the paper, let
\begin{equation}\label{v}
v(x)= C_{n,-\beta}|x|^{\beta-n}\ast u^{p} = C_{n,-\beta} \int _{R^{n}} \frac{u^{p}(y)}{|x-y|^{n-\beta}}dy,
\end{equation}
then we have
\begin{equation*}
 (-\Delta)^{\frac{\beta}{2}}v= u^{p}  \,\,\,\,\,\,\,\,\,\mbox{in}\,\,\,\,\, \mathbb{R}^{n}.
 \end{equation*}
Therefore, \eqref{Big1} is equivalent to
\begin{equation}\label{Big2}
\left\{\begin{array}{r@{\ \ }c@{\ \ }ll}
\left(-\Delta\right)^{\frac{\alpha}{2}} u(x) + u(x)&=& v(x)u^{p-1}(x), & \ \ x\in \mathbb{R}^{n}\,, \\[0.05cm]
\left(-\Delta\right)^{\frac{\beta}{2}} v(x) &=& u^{p}(x), & \ \ x\in \mathbb{R}^{n}\,, \\[0.05cm]
u(x) > 0, && v(x)>0, & \ \ x\in\mathbb{R}^{n}\,.
\end{array}\right.
\end{equation}
Hence,the project to study \eqref{Big1} turns out to consider its equivalent problem \eqref{Big2}.

Before stating our main results, we first give some basic notations. Let
\begin{equation*}
T_{\lambda} =\{x \in \mathbb{R}^{n}|\, x_1=\lambda, \mbox{ for some } \lambda\in \mathbb{R}\}
\end{equation*}
be the moving plane,\quad $\Sigma_{\lambda} =\{x \in \mathbb{R}^{n} | \, x_1<\lambda\}$ be the region to the left of the plane,$\quad \widetilde{\Sigma}_{\lambda}=\mathbb{R}^{n}\setminus\Sigma_{\lambda}$,\quad and $x^{\lambda}=(2\lambda-x_{1},x_{2},\cdot\cdot\cdot,x_{n})$ be the reflection of the point $x=(x_{1},x_{2},\cdot\cdot\cdot,x_{n})=(x_{1},x')$ about the plane $T_{\lambda}$.

To compare the values of $u(x)$ with $u(x^{\lambda})$ and $v(x)$ with $v(x^{\lambda})$, we denote
\begin{equation*}
\left\{\begin{array}{r@{\ \ }c@{\ \ }ll}
U_{\lambda} (x)& :=&u(x^{\lambda}) - u(x), \\[0.05cm]
V_{\lambda} (x)& :=&v(x^{\lambda}) - v(x).
\end{array}\right.
\end{equation*}
Obviously, $U_{\lambda}$ and $V_{\lambda}$ are anti-symmetric
functions, i.e., $U_{\lambda}(x^{\lambda})=-U_{\lambda}(x)$ and
$V_{\lambda}(x^{\lambda})=-V_{\lambda}(x)$. Furthermore,
$C$ denotes a constant whose value may be different from line to
line, and only the relevant dependence is specified in this paper.

 To be precise, we state our main result as follows.

\begin{theorem}\label{Th1}
Suppose that $u\in C^{1,1}_{loc}\cap L_{\alpha}$ and $v\in C^{1,1}_{loc}\cap L_{\beta}$ is a positive solution pair of \eqref{Big2} and
 \begin{equation}\label{uv-decay}
 \frac{C_{1}}{|x|^{\gamma}}\leq u(x)\leq \frac{C_{2}}{|x|^{\gamma}}   \mbox{ for } |x| \mbox{ sufficiently large},\, \,\,\, \underset{|x| \rightarrow \infty}{\lim} u^{p-2}(x) v(x) = a <\frac{C_{1}}{C_{2}p},
 \end{equation}
 where $C_{1}, C_{2}, a, \gamma $ are constants and $\gamma$ satisfies $\gamma >\frac{\beta}{p-1}$.\quad Then $u$ must be radially symmetric and monotone decreasing about some point in $\mathbb{R}^n$.
\end{theorem}

\begin{remark}
Since the Kelvin transform is not valid for the Schr$ \ddot{\textrm{o}}$dinger equations because of the presence of the term $u$, we need to impose the additional assumptions on the behavior of $u$ and $v$ at infinity.
\end{remark}

In order to prove Theorem \ref{Th1}\,, we need to obtain the
following two auxiliary conclusions, which are key ingredients
in the direct application of the method of moving planes.
\begin{theorem}\label{Th2}(\textbf{Decay at infinity})
Let $\Omega$ be an unbounded region in $\Sigma_{\lambda}$. Assume that $U_{\lambda}(x)\in L_{\alpha}\cap C^{1,1}_{loc}(\Omega)$ and $V_{\lambda}(x)\in L_{\beta}\cap C^{1,1}_{loc}(\Omega)$ are lower semi-continuous on $\overline{\Omega}$, which satisfy
\begin{equation}\label{decay-model}
\left\{\begin{array}{r@{\ \ }c@{\ \ }ll}
&&\left(-\Delta\right)^{\frac{\alpha}{2}}U_{\lambda}(x)+ C_{2}(x) U_{\lambda}(x)+C_{3}(x) V_{\lambda}(x)\geq0, & \ \ x\in\Omega\,, \\[0.05cm]
&&\left(-\Delta\right)^{\frac{\beta}{2}}V_{\lambda}(x)+C_{1}(x)U_{\lambda}(x)\geq 0, & \ \ x\in\Omega\,, \\[0.05cm]
&&U_{\lambda}(x) \geq 0,\,  V_{\lambda}(x)\geq0, & \ \ x\in\Sigma_{\lambda}\backslash\Omega\,, \\[0.05cm]
&&U_{\lambda}(x^{\lambda})=-U_{\lambda}(x),\,\,\,  V_{\lambda}(x^{\lambda})=-V_{\lambda}(x), & \ \ x\in\Sigma_{\lambda}\,,
\end{array}\right.
\end{equation}
with $C_{1}(x),\,C_{3}(x)<0$,\,\,\,\,$C_{2}(x)>0$,\,\,and
\begin{equation}\label{decay-C1(x)}
 C_{1}(x)\sim \frac{1}{|x|^{m}} \quad \mbox{for}\,  |x| \, \mbox{sufficiently large}, \quad \forall \,m > \beta,
\end{equation}
\begin{equation}
\label{decay-C3(x)}
\underset{|x| \rightarrow \infty}{\lim} C_{3}(x)=0.
\end{equation}
Then there exists a constant $R_{0} > 0$ such that if
\begin{equation}\label{decay-UV}
  U_{\lambda}(x^{0})=\min_{\Omega} U_{\lambda}(x)<0,\quad  V_{\lambda}(x^{1})=\min_{\Omega} V_{\lambda}(x)<0,
\end{equation}
then at least one of $x^{0}$ and $x^{1}$ satisfies
\begin{equation}\label{decay-result}
  |x|\leq R_{0}.
\end{equation}
\end{theorem}

\begin{theorem}\label{Th3}(\textbf{Narrow region principle})
Let $\Omega$ be a bounded narrow region in $\Sigma_{\lambda}$, such that it is contained in $\left\{x\mid \lambda-\delta<x_{1}<\lambda\right\}$ with a small $\delta$. Assume that $u\in L_{\alpha}\cap C^{1,1}_{loc}(\Omega)$ and $v\in L_{\beta}\cap C^{1,1}_{loc}(\Omega)$ are lower semi-continuous on $\overline{\Omega}$, which satisfy
\begin{equation}\label{NRP-model}
\left\{\begin{array}{r@{\ \ }c@{\ \ }ll}
&&\left(-\Delta\right)^{\frac{\alpha}{2}}U_{\lambda}(x) + C_{2}(x)U_{\lambda}(x)+C_{3}(x)V_{\lambda}(x)\geq0, & \ \ x\in\Omega\,, \\[0.05cm]
&&\left(-\Delta\right)^{\frac{\beta}{2}}V_{\lambda}(x)+C_{1}(x)U_{\lambda}(x)\geq 0, & \ \ x\in\Omega\,, \\[0.05cm]
&&U_{\lambda}(x) \geq 0,\,  V_{\lambda}(x)\geq0, & \ \ x\in\Sigma_{\lambda}\backslash\Omega\,, \\[0.05cm]
&&U_{\lambda}(x^{\lambda})=-U_{\lambda}(x),\,  V_{\lambda}(x^{\lambda})=-V_{\lambda}(x), & \ \ x\in\Sigma_{\lambda}\,,
\end{array}\right.
\end{equation}
where $C_{1}(x)$, $C_{2}(x)$ and $C_{3}(x)$ are bounded from below
in $\Omega$ and $C_{1}(x),\,C_{3}(x)<0$. Then for sufficiently small $\delta$, we have
\begin{equation}\label{NRP-r1}
  U_{\lambda}(x),\,V_{\lambda}(x)\geq0,\quad x\in\Omega.
\end{equation}
Moreover, if\, $U_{\lambda}(x)=0$ or $V_{\lambda}(x)=0$ at some point $x$ in $\Omega$, then
\begin{equation}\label{NRP-r2}
  U_{\lambda}(x)=V_{\lambda}(x)\equiv0 \quad \mbox{almost everywhere in} \
  \mathbb{R}^{n}.
\end{equation}
The above conclusions hold for an unbounded narrow region $\Omega$ if we further assume that
\begin{equation*}
\underset{|x| \rightarrow \infty}{\underline{\lim}} U_{\lambda}(x),\,V_{\lambda}(x)\geq0.
\end{equation*}
\end{theorem}

The remainder of this paper is organized as follows. The proof of Theorem \ref{Th2} and Theorem \ref{Th3} is given in section \ref{section2}. Then we prove Theorem \ref{Th1} in the last section.

\section{Decay at Infinity and Narrow Region Principle}\label{section2}

In this section, we give the proof of Theorem \ref{Th2} and Theorem
\ref{Th3} which are essential in the subsequent section to derive Theorem \ref{Th1}\,.

\begin{proof}[Proof of Theorem \ref{Th2}\,](\textbf{Decay at Infinity})
We argue by contradiction, if \eqref{decay-result} is violated, then from the definition of the fractional Laplacian \eqref{Frac-Laplacian} and a direct calculation
\begin{eqnarray}\label{d-1}
  (-\Delta)^{\frac{\beta}{2}}V_{\lambda}(x^{1}) &=& C_{n,\beta} PV \int_{\mathbb{R}^{n}}\frac{V_{\lambda}(x^{1})-V_{\lambda}(y)}{|x^{1}-y|^{n+\beta}}dy \nonumber\\
   &=&   C_{n,\beta} PV \int_{\Sigma_{\lambda}}\frac{V_{\lambda}(x^{1})-V_{\lambda}(y)}{|x^{1}-y|^{n+\beta}}dy+C_{n,\beta} PV \int_{\widetilde{\Sigma}_{\lambda}}\frac{V_{\lambda}(x^{1})-V_{\lambda}(y)}{|x^{1}-y|^{n+\beta}}dy\nonumber\\
   &=&  C_{n,\beta} PV \int_{\Sigma_{\lambda}}\frac{V_{\lambda}(x^{1})-V_{\lambda}(y)}{|x^{1}-y|^{n+\beta}}dy+ C_{n,\beta} PV \int_{\Sigma_{\lambda}}\frac{V_{\lambda}(x^{1})-V_{\lambda}(y^{\lambda})}{|x^{1}-y^{\lambda}|^{n+\beta}}dy\nonumber\\
   &=&  C_{n,\beta} PV
   \int_{\Sigma_{\lambda}}\frac{V_{\lambda}(x^{1})-V_{\lambda}(y)}{|x^{1}-y|^{n+\beta}}dy+ C_{n,\beta} PV \int_{\Sigma_{\lambda}}\frac{V_{\lambda}(x^{1})+V_{\lambda}(y)}{|x^{1}-y^{\lambda}|^{n+\beta}}dy\nonumber\\
  &\leq & C_{n,\beta}
  \int_{\Sigma_{\lambda}}\frac{V_{\lambda}(x^{1})-V_{\lambda}(y)}{|x^{1}-y^{\lambda}|^{n+\beta}}+\frac{V_{\lambda}(x^{1})+V_{\lambda}(y)}{|x^{1}-y^{\lambda}|^{n+\beta}}dy\nonumber  \\
   &=&  C_{n,\beta}  \int_{\Sigma_{\lambda}}\frac{2V_{\lambda}(x^{1})}{|x^{1}-y^{\lambda}|^{n+\beta}}dy,
\end{eqnarray}
where $\widetilde{\Sigma}_{\lambda}=\mathbb{R}^{n}\backslash\Sigma_{\lambda}$.\\
Now we estimate the integral term in \eqref{d-1} for fixed $\lambda$ and $x^{1}\in\Sigma_{\lambda}$ with $|x^{1}|$ sufficiently large. Let $x^{2}=\left(3|x^{1}|+x_{1}^{1},(x^{1})'\right)$, then $B_{|x^{1}|}(x^{2})\subset\widetilde{\Sigma}_{\lambda}$. Then we have
\begin{eqnarray}\label{d-2}
   \int_{\Sigma_{\lambda}}\frac{1}{|x^{1}-y^{\lambda}|^{n+\beta}}dy &=& \int_{\widetilde{\Sigma}_{\lambda}}\frac{1}{|x^{1}-y|^{n+\beta}}dy \nonumber\\
   &\geq& \int_{B_{|x^{1}|}(x^{2})}\frac{1}{|x^{1}-y|^{n+\beta}}dy \nonumber\\
  &\geq& \int_{B_{|x^{1}|}(x^{2})}\frac{1}{4^{n+\beta}|x^{1}|^{n+\beta}}dy \nonumber\\
  &=&\frac{\omega_{n}}{4^{n+\beta}|x^{1}|^{\beta}}.
\end{eqnarray}
Then inserting \eqref{d-2} into \eqref{d-1}\,, we obtain
\begin{equation}\label{d-3}
  (-\Delta)^{\frac{\beta}{2}}V_{\lambda}(x^{1})\leq\frac{C}{|x^{1}|^{\beta}}V_{\lambda}(x^{1})<0.
\end{equation}
Combining \eqref{d-3} with \eqref{decay-model} and $C_{1}(x)<0$, we derive
\begin{equation}\label{d-4}
  U_{\lambda}(x^{1})<0,
\end{equation}
and
\begin{equation}\label{d-5}
  V_{\lambda}(x^{1})\geq-CC_{1}(x^{1})|x^{1}|^{\beta}U_{\lambda}(x^{1}).
\end{equation}
In terms of \eqref{d-4},\eqref{uv-decay} and the lower semi-continuity of $U_{\lambda}$ on $\overline{\Omega}$, we can show that there exists $x^{0}\in\Omega$ such that
\begin{equation*}
   U_{\lambda}(x^{0})=\min_{\Sigma_{\lambda}} U_{\lambda}(x)<0.
\end{equation*}
From a similar argument as in \eqref{d-1}\eqref{d-2}, we can show
\begin{equation}\label{d-6}
   (-\Delta)^{\frac{\alpha}{2}} U_{\lambda}(x^{0})\leq \frac{C}{|x^{0}|^{\alpha}}U_{\lambda}(x^{0})
\end{equation}
From \eqref{decay-model},\eqref{d-5} and \eqref{d-6}, we can deduce
\begin{eqnarray}
  0&\leq&(-\Delta)^{\frac{\alpha}{2}}U_{\lambda}(x^{0})+C_{2}(x^{0})U_{\lambda}(x^{0})+C_{3}(x^{0})V_{\lambda}(x^{0})\nonumber\\
   &\leq& \frac{C}{|x^{0}|^{\alpha}}U_{\lambda}(x^{0})+C_{2}(x^{0})U_{\lambda}(x^{0})+C_{3}(x^{0})V_{\lambda}(x^{1})\nonumber\\
   &\leq&
   \frac{C}{|x^{0}|^{\alpha}}U_{\lambda}(x^{0})+C_{2}(x^{0})U_{\lambda}(x^{0})-CC_{3}(x^{0})C_{1}(x^{1})|x^{1}|^{\beta}U_{\lambda}(x^{1})
   \nonumber \\
   &<&\,\,0 \nonumber
\end{eqnarray}
for sufficiently large $|x^{0}|$ and $|x^{1}|$. The last inequality follows from assumptions \eqref{decay-C1(x)} and \eqref{decay-C3(x)}. This is a contradiction. Hence, the relation \eqref{decay-result} must be valid for at least one of $x^{0}$ and $x^{1}$. The proof of Theorem \ref{Th2} is completed.
\end{proof}

Now we start the proof of Theorem \ref{Th3} as follows.
\begin{proof}[Proof of Theorem \ref{Th3}\,](\textbf{Narrow Region Principle}) The proof goes by contradiction. Without loss of generality, we assume that there exists a $\overline{x}\in \overline{\Omega}$ such that
$$V_{\lambda}(\overline{x})=\min_{\overline{\Omega}} V_{\lambda}(x)<0.$$
By the definition of $(-\Delta)^{\frac{\beta}{2}}$, we have
\begin{eqnarray*}
  (-\Delta)^{\frac{\beta}{2}}V_{\lambda}(\overline{x}) &=& C_{n,\beta} PV \int_{\mathbb{R}^{n}}\frac{V_{\lambda}(\overline{x})-V_{\lambda}(y)}{|\overline{x}-y|^{n+\beta}}dy \nonumber\\
   &=&   C_{n,\beta} PV \int_{\Sigma_{\lambda}}\frac{V_{\lambda}(\overline{x})-V_{\lambda}(y)}{|\overline{x}-y|^{n+\beta}}dy+C_{n,\beta} PV \int_{\widetilde{\Sigma}_{\lambda}}\frac{V_{\lambda}(\overline{x})-V_{\lambda}(y)}{|\overline{x}-y|^{n+\beta}}dy\nonumber\\
   &=&  C_{n,\beta} PV \int_{\Sigma_{\lambda}}\frac{V_{\lambda}(\overline{x})-V_{\lambda}(y)}{|\overline{x}-y|^{n+\beta}}dy+ C_{n,\beta} PV \int_{\Sigma_{\lambda}}\frac{V_{\lambda}(\overline{x})-V_{\lambda}(y^{\lambda})}{|\overline{x}-y^{\lambda}|^{n+\beta}}dy\nonumber\\
   &=&  C_{n,\beta} PV
   \int_{\Sigma_{\lambda}}\frac{V_{\lambda}(\overline{x})-V_{\lambda}(y)}{|\overline{x}-y|^{n+\beta}}dy+ C_{n,\beta} PV \int_{\Sigma_{\lambda}}\frac{V_{\lambda}(\overline{x})+V_{\lambda}(y)}{|\overline{x}-y^{\lambda}|^{n+\beta}}dy\nonumber\\
  &\leq & C_{n,\beta}
  \int_{\Sigma_{\lambda}}\frac{V_{\lambda}(\overline{x})-V_{\lambda}(y)}{|\overline{x}-y^{\lambda}|^{n+\beta}}+\frac{V_{\lambda}(\overline{x})+V_{\lambda}(y)}{|\overline{x}-y^{\lambda}|^{n+\beta}}dy\nonumber  \\
   &=&  C_{n,\beta}  \int_{\Sigma_{\lambda}}\frac{2V_{\lambda}(\overline{x})}{|\overline{x}-y^{\lambda}|^{n+\beta}}dy,
\end{eqnarray*}
Let $D:=\left(B_{2\delta}(\overline{x})\backslash B_{\delta}(\overline{x})\right)\cap \widetilde{\sum}_{\lambda}$\,, then
\begin{eqnarray}\label{n-1}
  \int_{\Sigma_{\lambda}}\frac{1}{|\overline{x}-y^{\lambda}|^{n+\beta}}dy &=& \int_{\widetilde{\Sigma}_{\lambda}}\frac{1}{|\overline{x}-y|^{n+\beta}}dy \nonumber\\
   &\geq& \int_{D}\frac{1}{|\overline{x}-y|^{n+\beta}}dy \nonumber \\
   &\geq& C\int_{B_{2\delta}(\overline{x})\backslash B_{\delta}(\overline{x})}\frac{1}{|\overline{x}-y|^{n+\beta}}dy \nonumber \\
   &\sim& \frac{C}{\delta^{\beta}}.
\end{eqnarray}
Thus,
\begin{equation}\label{n-2}
  (-\Delta)^{\frac{\beta}{2}}V_{\lambda}(\overline{x})\leq\frac{CV_{\lambda}(\overline{x})}{\delta^{\beta}}<0.
\end{equation}
Combining this with \eqref{NRP-model} and $C_{1}(x)<0$, we have
\begin{equation}\label{n-3}
U_{\lambda}(\overline{x})<0
\end{equation}
and
\begin{equation}\label{n-4}
-C_{1}(\overline{x})U_{\lambda}(\overline{x})\leq\frac{CV_{\lambda}(\overline{x})}{\delta^{\beta}}.
\end{equation}
From\eqref{n-3}, we know that there exists a $\widetilde{x}$ such that
\begin{equation*}
U_{\lambda}(\widetilde{x})= \min_{\Omega}U_{\lambda}(x)<0.
\end{equation*}
Similar to \eqref{n-2}, we can derive that
\begin{equation}\label{n-5}
(-\Delta)^{\frac{\alpha}{2}}U_{\lambda}(\widetilde{x})\leq\frac{CU_{\lambda}(\widetilde{x})}{\delta^{\alpha}}<0.
\end{equation}
By \eqref{NRP-model} and $C_{3}(x)<0$, for $\delta$ sufficiently small, we have
\begin{eqnarray*}
  0&\leq&(-\Delta)^{\frac{\alpha}{2}}U_{\lambda}(\widetilde{x})+C_{2}(\widetilde{x})U_{\lambda}(\widetilde{x})+C_{3}(\widetilde{x})V_{\lambda}(\widetilde{x})\nonumber\\
   &\leq&\frac{CU_{\lambda}(\widetilde{x})}{\delta^{\alpha}}+C_{2}(\widetilde{x})U_{\lambda}(\widetilde{x})+C_{3}(\widetilde{x})V_{\lambda}(\widetilde{x})  \nonumber\\
   &\leq&\frac{CU_{\lambda}(\widetilde{x})}{\delta^{\alpha}}+C_{2}(\widetilde{x})U_{\lambda}(\widetilde{x})-C_{3}(\widetilde{x})C_{1}(\overline{x})\cdot\frac{1}{C}\delta^{\beta}U_{\lambda}(\overline{x}) \nonumber\\
   &\leq& \frac{CU_{\lambda}(\widetilde{x})}{\delta^{\alpha}}+C_{2}(\widetilde{x})U_{\lambda}(\widetilde{x})-C_{3}(\widetilde{x})C_{1}(\overline{x})\cdot\frac{1}{C}\delta^{\beta}U_{\lambda}(\widetilde{x}) \nonumber\\
   &\leq&
   \frac{CU_{\lambda}(\widetilde{x})}{\delta^{\alpha}}[1+\frac{C_{2}(\widetilde{x})\delta^{\alpha}}{C}-\frac{C_{3}(\widetilde{x})C_{1}(\overline{x})}{C^{2}}\delta^{\alpha}\delta^{\beta}]
   \nonumber\\
   &<&0
\end{eqnarray*}
which is a contradiction.
Therefore, \eqref{NRP-r1} is true.\\
To prove \eqref{NRP-r2}, we suppose there exists $x_{0}\in\Omega$ such that $\phi(x_{0})=0$.
\begin{eqnarray}\label{n-6}
  \textmd{Then,} \quad (-\Delta)^{\frac{\beta}{2}}V_{\lambda}(x_{0}) &=& C_{n,\beta} PV \int_{\mathbb{R}^{n}}\frac{V_{\lambda}(x_{0})-V_{\lambda}(y)}{|x_{0}-y|^{n+\beta}}dy \nonumber\\
   &=&   C_{n,\beta} PV \int_{\Sigma_{\lambda}}\frac{-V_{\lambda}(y)}{|x_{0}-y|^{n+\beta}}dy+C_{n,\beta} PV \int_{\widetilde{\Sigma}_{\lambda}}\frac{-V_{\lambda}(y)}{|x_{0}-y|^{n+\beta}}dy\nonumber\\
   &=&  C_{n,\beta} PV \int_{\Sigma_{\lambda}}\frac{-V_{\lambda}(y)}{|x_{0}-y|^{n+\beta}}dy+ C_{n,\beta} PV \int_{\Sigma_{\lambda}}\frac{-V_{\lambda}(y^{\lambda})}{|x_{0}-y^{\lambda}|^{n+\beta}}dy\nonumber\\
   &=&  C_{n,\beta} PV
   \int_{\Sigma_{\lambda}}\frac{-V_{\lambda}(y)}{|x_{0}-y|^{n+\beta}}dy+ C_{n,\beta} PV \int_{\Sigma_{\lambda}}\frac{V_{\lambda}(y)}{|x_{0}-y^{\lambda}|^{n+\beta}}dy\nonumber\\
  &=& C_{n,\beta} PV
  \int_{\Sigma_{\lambda}}(\frac{1}{|x_{0}-y^{\lambda}|^{n+\beta}}-\frac{1}{|x_{0}-y|^{n+\beta}})\,V_{\lambda}(y)dy.
\end{eqnarray}
From $\frac{1}{|x_{0}-y^{\lambda}|^{n+\beta}}-\frac{1}{|x_{0}-y|^{n+\beta}}<0$, \,we know that if\, $V_{\lambda}(x) \neq 0 \,\,in\,\, \Sigma_{\lambda}$, \,\eqref{n-6} implies that
\begin{equation*}
(-\Delta)^{\frac{\beta}{2}}V_{\lambda}(x_{0})<0.
\end{equation*}
Combining this with \eqref{NRP-model} and $C_{1}(x) < 0$, we can derive
\begin{equation*}
U_{\lambda}(x_{0})<0,
\end{equation*}
which is a contradiction with \eqref{NRP-r1}.\\
Therefore $V_{\lambda}(x_{0})$ is identically 0 in $\Sigma_{\lambda}$. Since
\begin{equation*}
V_{\lambda}(x^{\lambda})=-V_{\lambda}(x),\,\,x\in \Sigma_{\lambda},
\end{equation*}
it shows that
\begin{equation*}
V_{\lambda}(x)=0,\,\,x\in \mathbb{R}^{n},
\end{equation*}
then
\begin{equation*}
(-\Delta)^{\frac{\beta}{2}}V_{\lambda}(x)=0,\,\,x\in \mathbb{R}^{n}.
\end{equation*}
From \eqref{NRP-model} and $C_{1}(x)<0$, we know
\begin{equation*}
U_{\lambda}(x)\leq 0, \,\,\,x\in \Sigma_{\lambda}.
\end{equation*}
We already proved
\begin{equation*}
U_{\lambda}(x)\geq 0, \,\,\,x\in \Sigma_{\lambda}.
\end{equation*}
Consequently it must be true that
\begin{equation*}
U_{\lambda}(x)= 0, \,\,\,x\in \Sigma_{\lambda}.
\end{equation*}
Combining this with the fact
\begin{equation*}
U_{\lambda}(x^{\lambda})= -U_{\lambda}(x), \,\,\,x\in \Sigma_{\lambda},
\end{equation*}
we get
\begin{equation*}
U_{\lambda}(x)\equiv 0, \,\,\,x\in \mathbb{R}^{n}.
\end{equation*}
From a similar argument, we can show if $U_{\lambda}(x)$ is 0 at one point in $\Sigma_{\lambda}$, then $V_{\lambda}(x)$ and $U_{\lambda}(x)$ are identically 0 in $\mathbb{R}^{n}$.\\
For the unbounded narrow region $\Omega$, the condition
\begin{equation*}
\underset{|x| \rightarrow \infty}{\underline{\lim}} U_{\lambda}(x),\,V_{\lambda}(x)\geq0
\end{equation*}
guarantees that the negative minimum of $U_{\lambda}(x)$ and $V_{\lambda}(x)$ must be attained at some point $x^{0}$ and $x^{1}$ respectively, then we can derive the similar contradiction as above.

This completes the proof of Theorem \ref{Th3} .
\end{proof}

\section{Radial Symmetry of Positive Solutions}\label{section3}
In this section, we prove the radial symmetry of positive solutions to \eqref{Big2}.

\begin{proof}[Proof of Theorem \ref{Th1}]Choosing a direction to be $x_{1}$-direction, then we divide the proof into two steps.

\noindent \textup{\textbf{Step 1.}} Start moving the plane $T_{\lambda}$ from $-\infty$ to the right along the $x_{1}$-axis. We will claim that
\begin{equation}\label{r-1}
  U_{\lambda}(x),\, V_{\lambda}(x)\geq0, \quad \forall x\in\Sigma_{\lambda}
\end{equation}
for sufficiently negative $\lambda$.

If \eqref{r-1} is violated, without loss of generality, we assume that there exists an $\overline{x} \in \Sigma_\lambda$ such that
$$ U_{\lambda}(\overline{x}) = \min_{\Sigma_\lambda} U_\lambda(x) < 0 .$$
Similar estimates as \eqref{d-1}-\eqref{d-3} in the proof of Theorem \ref{Th2} yield that
\begin{equation}\label{r-2}
  (-\Delta)^{\frac{\alpha}{2}}U_{\lambda}(\overline{x})\leq C \,\frac{U_{\lambda}(\overline{x})}{|\overline{x}|^{\alpha}}<0.
\end{equation}
Now we show that
\begin{equation}\label{r-3}
V_{\lambda}(\overline{x})<0.
\end{equation}
Indeed, if not, $V_{\lambda}(\overline{x})\geq0.$
\,\,Then by \eqref{uv-decay} and mean value principle, \begin{eqnarray}\label{r-4}
  (-\Delta)^{\frac{\alpha}{2}}U_{\lambda}(\overline{x})&=&
  (-\Delta)^{\frac{\alpha}{2}}u(\overline{x}^{\lambda})-(-\Delta)^{\frac{\alpha}{2}}u(\overline{x})\nonumber\\ &=& v(\overline{x}^{\lambda})u^{p-1}(\overline{x}^{\lambda})-u(\overline{x}^{\lambda})-v(\overline{x})u^{p-1}(\overline{x}) + u(\overline{x}) \nonumber\\
  &=&
  v(\overline{x}^{\lambda})u^{p-1}(\overline{x}^{\lambda})-v(\overline{x})u^{p-1}(\overline{x}^{\lambda})+v(\overline{x})u^{p-1}(\overline{x}^{\lambda})-v(\overline{x})u^{p-1}(\overline{x})-U_{\lambda}(\overline{x})
    \nonumber\\
  &=&
  u^{p-1}(\overline{x}^{\lambda})V_{\lambda}(\overline{x})+v(\overline{x})\left( u^{p-1}(\overline{x}^{\lambda})-u^{p-1}(\overline{x})\right)-U_{\lambda}(\overline{x})
   \nonumber\\
   &\geq&
   u^{p-1}(\overline{x}^{\lambda})V_{\lambda}(\overline{x})+ v(\overline{x})\frac{1}{C_{1}}|\overline{x}|^{\gamma}u(\overline{x}) \left( u^{p-1}(\overline{x}^{\lambda})-u^{p-1}(\overline{x})\right) -U_{\lambda}(\overline{x}) \nonumber \\
   &\geq& u^{p-1}(\overline{x}^{\lambda})\,V_{\lambda}(\overline{x})+ \frac{1}{C_{1}} v(\overline{x})|\overline{x}| ^{\gamma}\left(u^{p}( \overline{x}^{\lambda})-u^{p}(\overline{x}))-U_{\lambda}(\overline{x}\right)        \nonumber \\
   &=& u^{p-1}(\overline{x}^{\lambda})\,V_{\lambda}(\overline{x})+\frac{1}{C_{1}}v(\overline{x})\,|\overline{x}|^{\gamma}\,p\,\xi^{p-1}\left( u(\overline{x}^{\lambda})-u(\overline{x})\right)-U_{\lambda}(\overline{x}) \nonumber \\
   &\geq& u^{p-1}(\overline{x}^{\lambda})\,V_{\lambda}(\overline{x})+ \frac{1}{C_{1}}v(\overline{x})|\overline{x}|^{\gamma}\,p\,u^{p-1}(\overline{x})\,U_{\lambda}(\overline{x}) -U_{\lambda}(\overline{x}) \nonumber \\
   &=& u^{p-1}(\overline{x}^{\lambda})\,V_{\lambda}(\overline{x})+\left(\frac{1}{C_{1}}\,p\,u^{p-1}(\overline{x})v(\overline{x})|\overline{x}|^{\gamma}-1\right)\,U_{\lambda}(\overline{x}).
\end{eqnarray}
 where $\xi\in \left(u(\overline{x}^{\lambda}),\,u(\overline{x})\right)$.

 Then from condition \eqref{uv-decay}, we know that for $\lambda$ sufficiently negative,
 \begin{eqnarray*}
 \frac{1}{C_{1}}\,p\,u^{p-1}(\overline{x})v(\overline{x})|\overline{x}|^{\gamma}-1
 &\leq& \frac{1}{C_{1}}\,p\,u^{p-1}(\overline{x})v(\overline{x})\frac{C_{2}}{u}-1 \nonumber \\
 &=& \frac{C_{2}}{C_{1}}\,p\,u^{p-2}(\overline{x})v(\overline{x})-1 \nonumber \\
 &<& 0
 \end{eqnarray*}
As a matter of fact, $(-\Delta)^{\frac{\alpha}{2}}U_{\lambda}(\overline{x})>0$, contradiction.
 Thus, \eqref{r-3} must hold. So we can conclude there exists an $\widetilde{x} \in \Sigma_\lambda$ such that
$$ V_{\lambda}(\widetilde{x}) = \min_{\Sigma_\lambda} V_\lambda < 0 .$$
Next, we verify that
\begin{equation}\label{r-5}
U_{\lambda}(\widetilde{x})<0.
\end{equation}
If not, on one hand, similar to \eqref{d-3}, we have
\begin{equation*}
  (-\Delta)^{\frac{\beta}{2}}V_{\lambda}(\widetilde{x})\leq\frac{C}{|\widetilde{x}|^{\beta}}V_{\lambda}(\widetilde{x})<0.
\end{equation*}
On the other hand, it follows from mean value theorem that
\begin{eqnarray*}
  (-\Delta)^{\frac{\beta}{2}}V_{\lambda}(\widetilde{x}) &=&
   (-\Delta)^{\frac{\beta}{2}}v(\widetilde{x}^{\lambda})-(-\Delta)^{\frac{\beta}{2}}v(\widetilde{x})\\
   &=&u^{p}(\widetilde{x}^{\lambda})-u^{p}(\widetilde{x}) \\
   &=& p\vartheta^{p-1}U_{\lambda}(\widetilde{x})\\
   &\geq& pu^{p-1}(\widetilde{x})U_{\lambda}(\widetilde{x})\\ &\geq& 0,
\end{eqnarray*}
where $\vartheta\in\left(u(\widetilde{x}),\,u(\widetilde{x}^{\lambda})\right)$. This contradiction deduces \eqref{r-5}.

Thus, in terms of the above estimates and mean value theorem, we can derive
\begin{eqnarray}\label{r-6}
   (-\Delta)^{\frac{\alpha}{2}}U_{\lambda}(\overline{x}) &\geq&
   u^{p-1}(\overline{x}^{\lambda})V_{\lambda}(\overline{x})+\left(\frac{1}{C_{1}}pu^{p-1}(\overline{x})v(\overline{x})|\overline{x}|^{\gamma}-1\right)U_{\lambda}(\overline{x})
   \nonumber \\
   &\geq&  u^{p-1}(\overline{x})V_{\lambda}(\overline{x})+\left(\frac{1}{C_{1}}pu^{p-1}(\overline{x})v(\overline{x})|\overline{x}|^{\gamma}-1\right)U_{\lambda}(\overline{x})
\end{eqnarray}
and
\begin{equation}\label{r-7}
   (-\Delta)^{\frac{\alpha}{2}}V_{\lambda}(\widetilde{x})
   \geq pu^{p-1}(\widetilde{x})U_{\lambda}(\widetilde{x}).
\end{equation}
Let
\begin{eqnarray*}
  C_{1}(\widetilde{x}) &=& -pu^{p-1}(\widetilde{x})\\
   &\sim& \frac{1}{|\widetilde{x}|^{\gamma(p-1)}},
\end{eqnarray*}
\begin{eqnarray*}
  C_{2}(\overline{x}) &=& 1-\frac{1}{C_{1}}pu^{p-1}(\overline{x})v(\overline{x})|\overline{x}|^{\gamma}>0,
\end{eqnarray*}
and
\begin{eqnarray*}
  C_{3}(\overline{x}) &=& -u^{p-1}(\overline{x}) \\
   &\sim& \frac{1}{|\overline{x}|^{\gamma(p-1)}}
\end{eqnarray*}
for sufficiently negative $\lambda$. Then in terms of Theorem \ref{Th2} , we know that one of $U_{\lambda}(x)$ and $V_{\lambda}(x)$ must be nonnegative in $\Sigma_{\lambda}$ for sufficiently negative $\lambda$. Without loss of generality, we can suppose that
\begin{equation}\label{r-8}
  U_{\lambda}(x)\geq 0, \quad x\in\Sigma_{\lambda}.
\end{equation}
To show that \eqref{r-8} also holds for $V_{\lambda}(x)$, we argue by contradiction again. If $V_{\lambda}(x)$ is negative at some point in $\Sigma_{\lambda}$, then there must exist an $x_{0} \in \Sigma_\lambda$ such that
$$ V_{\lambda}(x_{0}) = \min_{\Sigma_\lambda} V_\lambda < 0 .$$
Then from a similar argument, we can show
\begin{equation*}
  0>\frac{CV_{\lambda}(x_{0})}{|x_{0}|^{\beta}}\geq(-\Delta)^{\frac{\beta}{2}}V_{\lambda}(x_{0})\geq pu^{p-1}(x_{0}^{\lambda})U_{\lambda}(x_{0})\geq0.
\end{equation*}
 Hence, the proof of Step 1 is completed.

\noindent \textup{\textbf{Step 2.}}  Step1 provides a starting point, from which we can now move the plane $T_{\lambda}$ to the right along the $x_{1}$-axis until its limiting position as long as \eqref{r-1} holds. More precisely, let
 \begin{equation*}
 \lambda_0 := \sup \{ \lambda \mid U_\mu (x) \geq 0,\, V_\mu (x) \geq 0, \; x \in \Sigma_\mu, \; \mu \leq \lambda \},
 \end{equation*}
 then the behavior of $u$ and $v$ at infinity guarantee $\lambda_{0}<\infty$.

In the sequel, we claim that $u$ is symmetric about the limiting plane $T_{\lambda_0}$, that is to say
\begin{equation}\label{r-9}
 U_{\lambda_0}(x)= V_{\lambda_0}(x)  \equiv 0 , \;\; x \in \Sigma_{\lambda_0}.
 \end{equation}
By the definition of $\lambda_{0}$, we first show that either
$$U_{\lambda_0}(x)= V_{\lambda_0}(x)  \equiv 0 , \;\; x \in \Sigma_{\lambda_0},$$
or
$$U_{\lambda_0}(x),\,V_{\lambda_0}(x) > 0 , \;\; x \in \Sigma_{\lambda_0}.$$
To prove this, without loss of generality,
we assume there a point $\widetilde{x}\in\Sigma_{\lambda_{0}}$ such that
$$U_{\lambda_{0}}(\widetilde{x})= \min_{\Sigma_{\lambda_{0}}} U_{\lambda_{0}}=0,$$
then it must be revealed that
$$ U_{\lambda_0}(x) \equiv 0 , \;\; x \in \Sigma_{\lambda_0}.$$
If not, on one hand
\begin{eqnarray*}
  &&(-\Delta)^{\frac{\alpha}{2}}U_{\lambda_{0}}(\widetilde{x})\nonumber\\
   &=&  C_{n,\alpha}\,PV\int_{\mathbb{R}^{n}}\frac{U_{\lambda_{0}}(\widetilde{x})-U_{\lambda_{0}}(y)}{|\widetilde{x}-y|^{n+\alpha}}dy \nonumber\\
   &=& C_{n,\alpha}\,PV\int_{\mathbb{R}^{n}}\frac{-U_{\lambda_{0}}(y)}{|\widetilde{x}-y|^{n+\alpha}}dy\nonumber\\
   &=& C_{n,\alpha}\,PV\int_{\Sigma_{\lambda_{0}}}\frac{-U_{\lambda_{0}}(y)}{|\widetilde{x}-y|^{n+\alpha}}dy+C_{n,\alpha}\,PV\int_{\widetilde{\Sigma}_{\lambda_{0}}}\frac{-U_{\lambda_{0}}(y)}{|\widetilde{x}-y|^{n+\alpha}}dy
    \nonumber\\
    &<& C_{n,\alpha}\,PV\int_{\Sigma_{\lambda_{0}}}\frac{-U_{\lambda_{0}}(y)}{|\widetilde{x}-y^{\lambda}|^{n+\alpha}}dy+C_{n,\alpha}\,PV\int_{\Sigma_{\lambda_{0}}}\frac{U_{\lambda_{0}}(y)}{|\widetilde{x}-y^{\lambda}|^{n+\alpha}}dy
    \nonumber\\
   &=&0.
\end{eqnarray*}
On the other hand,
\begin{eqnarray*}
   (-\Delta)^{\frac{\alpha}{2}}U_{\lambda_{0}}(\widetilde{x})&=&
   (-\Delta)^{\frac{\alpha}{2}}u_{\lambda_{0}}(\widetilde{x})-(-\Delta)^{\frac{\alpha}{2}}u(\widetilde{x}) \nonumber \\ &=& v_{\lambda_{0}}(\widetilde{x})u_{\lambda_{0}}^{q-1}(\widetilde{x})-u_{\lambda_{0}}(\widetilde{x})-v(\widetilde{x})u^{q-1}(\widetilde{x})+u(\widetilde{x}) \nonumber \\
  &=&  v_{\lambda_{0}}(\widetilde{x})u_{\lambda_{0}}^{p-1}(\widetilde{x})-v(\widetilde{x})u^{p-1}(\widetilde{x})   \nonumber \\
  &=& V_{\lambda_{0}}(\widetilde{x})u^{p-1}(\widetilde{x})\geq0.
\end{eqnarray*}
 That's a contradiction. Then it follows from  \,\, $U_{\lambda_{0}}(x)=-U_{\lambda_{0}}(x^{\lambda_{0}})$ \,\,that
$$ U_{\lambda_0}(x) \equiv 0 , \;\; x \in \mathbb{R}^{n}.$$
Similarly, it can be easy to verify that
$$ V_{\lambda_0}(x) \equiv 0 , \;\; x \in \mathbb{R}^{n}.$$
Therefore, if \eqref{r-9} is violated, then
\begin{equation}\label{r-10}
U_{\lambda_0}(x),\,V_{\lambda_0}(x) > 0 , \;\; x \in \Sigma_{\lambda_0}.
\end{equation}

Next, we prove that the plane can still move further in this case. To be more rigorous, there exists $\varepsilon>0$ such that
\begin{equation}\label{r-11}
  U_{\lambda}(x),\,V_{\lambda}(x) \geq 0 , \;\; x \in \Sigma_{\lambda}
\end{equation}
for any $\lambda\in \left(\lambda_{0},\lambda_{0}+\varepsilon\right)$. This is a contradiction with the definition of $\lambda_{0}$, then the assertion \eqref{r-9} holds. We now focus to prove \eqref{r-11}.

From \eqref{r-10}, we have the following bounded away from zero estimate
\begin{equation*}
  U_{\lambda_{0}}(x),\,V_{\lambda_{0}}(x) \geq C_{\delta}>0 , \;\; x \in \overline{\Sigma_{\lambda_{0}-\delta}\cap B_{R_{0}}(0)}
\end{equation*}
for some $C_{\delta}>0$ and $R_{0}>0$. By the continuity of $U_{\lambda}(x)$ and $V_{\lambda}(x)$ with respect to $\lambda$, then there exists a positive constant $\varepsilon$ such that
\begin{equation*}
  U_{\lambda}(x),\,V_{\lambda}(x) \geq 0 , \;\; x \in \Sigma_{\lambda_{0}-\delta}\cap B_{R_{0}}(0)
\end{equation*}
for any $\lambda\in \left(\lambda_{0},\lambda_{0}+\varepsilon\right)$.

Furthermore, on the basis of Theorem \ref{Th2} , we know that if
 $$U_{\lambda}(x^0) = \min_{\Sigma_\lambda} U_\lambda < 0 \quad \mbox{and} \quad V_{\lambda}(x^1) = \min_{\Sigma_\lambda} V_\lambda < 0,$$
 then there exists a positive consant $R_{0}$ large enough such that one of $x^{0}$ and $x^{1}$ must be in $B_{R_{0}}(0)$. We may suppose $|x^{0}|<R_{0}$. Thus, we can derive
\begin{equation}\label{r-12}
  x^{0}\in\left(\Sigma_{\lambda}\backslash\Sigma_{\lambda_{0}-\delta}\right)\cap B_{R_{0}}(0).
\end{equation}
Next, we show that \eqref{r-12} also holds for $x^{1}$. Note that from a similar argument in the proof of \ref{Th3}, we have $$U_{\lambda}(x^{1})<0 \quad and \quad V_{\lambda}(x^{0})<0. $$ If $x^{1}\in \Sigma_{\lambda}\cap B^{c}_{R_{0}}(0)$, then from \eqref{d-3}, \eqref{Big2},\eqref{r-4}, mean value theorem and \eqref{n-5}, we have
\begin{eqnarray}\label{r-14}
  0>\frac{CV_{\lambda}(x^{1})}{|x^{1}|^{\beta}}
  &\geq&(-\Delta)^{\frac{\beta}{2}}V_{\lambda}(x^{1})= p\vartheta^{p-1}(x^{1})U_{\lambda}(x^{1})\nonumber\\
  &\geq& p\vartheta^{p-1}(x^{1})U_{\lambda}(x^{0})\geq
  pu^{p-1}(x^{1})U_{\lambda}(x^{0}),
\end{eqnarray}
and
\begin{eqnarray}\label{r-15}
  0 &>&  \frac{CU_{\lambda}(x^{0})}{(\delta+\varepsilon)^{\alpha}} \geq (-\Delta)^{\frac{\alpha}{2}}U_{\lambda}(x^{0}) \nonumber\\
   &\geq& u^{p-1}(x^{0})V_{\lambda}(x^{0})+\left( \frac{1}{C_{1}}p|x^{0}|^{\gamma}v(x^{0})u^{p-1}(x^{0})-1\right)U_{\lambda}(x^{0}) \nonumber\\
  &\geq& u^{p-1}(x^{0})V_{\lambda}(x^{1})+\left( \frac{1}{C_{1}}p|x^{0}|^{\gamma}v(x^{0})u^{p-1}(x^{0})-1\right)U_{\lambda}(x^{0}) .
\end{eqnarray}
for sufficiently small $\delta$, $\varepsilon$ and $\vartheta\in\left(u_{\lambda}(x^{1}),\,u(x^{1})\right)$.
Hence, a simple modification and combination of the above inequalities deduces that
\begin{eqnarray}\label{r-13}
  1 &\leq& (\delta+\varepsilon)^{\alpha}\left\{Cp|x^{1}|^{\beta}u^{p-1}(x^{0})u^{p-1}(x^{1})+C'\left[\frac{1}{C_{1}}p|x^{0}|^{\gamma}v(x^{0})u^{p-1}(x^{0})-1 \right] \right\}\nonumber\\
  &\leq& (\delta+\varepsilon)^{\alpha}\left\{Cp\frac{1}{|x^{1}|^{\gamma(p-1)-\beta}}u^{p-1}(x^{0})+C'\left[\frac{1}{C_{1}}p|x^{0}|^{\gamma}v(x^{0})u^{p-1}(x^{0})-1\right] \right\}
\end{eqnarray}
for sufficiently small $\delta$, $\varepsilon$ and large $R_{0}$. Since $\gamma(p-1)-\beta>0$,  $|x^{1}|>R_{0}$ and $x^{0}\in\left(\Sigma_{\lambda}\backslash\Sigma_{\lambda_{0}-\delta}\right)\cap B_{R_{0}}(0)$, then $Cp\frac{1}{|x^{1}|^{\gamma(p-1)-\beta}}u^{p-1}(x^{0})+C'\left[\frac{1}{C_{1}}p|x^{0}|^{\gamma}v(x^{0})u^{p-1}(x^{0})-1\right] $ is bounded, which implies that \eqref{r-13} must not be valid for sufficiently small $\delta$ and $\varepsilon$. This contradiction yields that
 $$x^{1}\in\left(\Sigma_{\lambda}\backslash\Sigma_{\lambda_{0}-\delta}\right)\cap B_{R_{0}}(0).$$
Finally, from \eqref{r-14} and \eqref{r-15}, we have
\begin{equation*}
  (-\Delta)^{\frac{\alpha}{2}}U_{\lambda}(x^{0})-u^{p-1}(x^{0})V_{\lambda}(x^{0})+\left(1-\frac{1}{C_{1}}p|x^{0}|^{\gamma}v(x^{0})u^{p-1}(x^{0})\right)U_{\lambda}(x^{0})\geq0
\end{equation*}
and
\begin{equation*}
  (-\Delta)^{\frac{\beta}{2}}V_{\lambda}(x^{1})-pu^{p-1}(x^{1})U_{\lambda}(x^{1})\geq0
\end{equation*}
Then combining with Theorem \ref{Th3}\,, we conclude that
\begin{equation*}
  U_{\lambda}(x),\,V_{\lambda}(x) \geq 0 , \;\; x \in \left(\Sigma_{\lambda}\backslash\Sigma_{\lambda_{0}-\delta}\right)\cap B_{R_{0}}(0)
\end{equation*}
for sufficiently small $\delta$ and $\varepsilon$. Hence, \eqref{r-11} holds.

Consequently, we must have
$$ U_{\lambda_0}(x)= V_{\lambda_0}(x)  \equiv 0 , \;\; x \in \Sigma_{\lambda_0}.$$
Since $x_1$ direction can be chosen arbitrarily, so we can conclude that the positive solution $u$ is radially symmetric and monotone decreasing with respect to some point in $\mathbb{R}^{n}$. This completes the proof of Theorem \ref{Th1} .
\end{proof}

\section*{Acknowledgments} This work is supported by the National Natural Science Foundation of China (NNSF Grant No. 11671414 and No. 11771218). The authors would like to deeply thank to Ma Lingwei for her selfless help.

\bibliography{bibliography}

\end{document}